\documentclass[12pt, reqno, a4paper]{amsart}
\usepackage{times}

\usepackage[utf8]{inputenc}
\usepackage[USenglish]{babel}
\usepackage{amsmath,amsthm,amssymb,amsfonts}
\usepackage{mathtools}
\usepackage{mathrsfs}
\usepackage{bbm}
\usepackage{booktabs}

\usepackage{indentfirst}
\usepackage{enumitem}
\usepackage{xcolor}

\usepackage{float}

\usepackage[breaklinks=true, bookmarksopenlevel=1, bookmarksdepth=2]{hyperref}
\hypersetup{
colorlinks,
   linkcolor={cyan!80!black},
   citecolor={cyan!80!black},
 urlcolor={cyan!80!black}
}

\allowdisplaybreaks

\newtheorem{thm}{Theorem}[section]
\newtheorem{cor}[thm]{Corollary}
\newtheorem{lem}[thm]{Lemma}

\theoremstyle{plain} 
\newcommand{\thistheoremname}{}
\newtheorem*{genericthm}{\thistheoremname}

\theoremstyle{definition}

\newtheorem{prob}[thm]{Problem}

\theoremstyle{remark}

\newtheorem*{ntt}{Notation}

\numberwithin{equation}{section}

\newcommand{\N}{\mathbb{N}}      
\newcommand{\Z}{\mathbb{Z}}      
\newcommand{\eps}{\varepsilon}   

\renewcommand{\pmod}[1]{         
  ~(\mathrm{mod}~#1)}

\newcommand{\ul}[1]{\mathbf{#1}} 

\usepackage[margin=1in]{geometry}

\restylefloat{table}

\setlist[itemize]{leftmargin=*}
\setlist[enumerate]{leftmargin=*}

\begin{document}

\title{Bounded asymptotic bases for linear forms}%
\author{Christian T\'afula}%
\address{Institute of Mathematics, Statistics\\
and Computer Science\\
University of S\~ao Paulo\\
Rua do Mat\~ao, 1010\\
S\~ao Paulo, SP 05508-090\\
Brazil}
\curraddr{}
\email{tafula@ime.usp.br}
\thanks{}


\begin{abstract}
 For a vector of positive integers $\mathbf{b} = (b_1,\ldots,b_h)$ with $\gcd(b_1,\ldots,b_h) = 1$, we study sets $A \subseteq \mathbb{N}$ for which every sufficiently large integer has a bounded positive number of representations
 \[ n = b_1 x_1 + \cdots + b_h x_h \qquad (x_1,\ldots,x_h\in A). \]
 We prove that such a set exists for every binary vector $\mathbf{b} \neq (1,1)$, and for some general higher-dimensional families, including $\mathbf{b} = (u_1, p^d u_2, \ldots, p^{(h-1)d} u_h)$ where $p\nmid u_1\cdots u_h$.
\end{abstract}

\maketitle

\section{Introduction}
 Let $h\geq 2$ be an integer, let $\ul{b}=(b_1,\ldots,b_h)$ be a vector of positive integers, and let $A \subseteq \N := \Z_{\geq 0}$. We consider the representation function
 \[ r_{A,\ul{b}}(n) := \#\{(x_1,\ldots,x_h)\in A^h ~|~ b_1 x_1 + \cdots + b_h x_h = n\}. \]
 We will always assume that $\ul{b}$ is \emph{primitive}, meaning that $\gcd(b_1,\ldots,b_h)=1$. This condition is necessary if every sufficiently large integer is to be represented. We say that $A$ is an \emph{asymptotic basis} for the linear form associated with $\ul{b}$ if $r_{A,\ul{b}}(n)>0$ for every sufficiently large $n$, and that it is a \emph{bounded asymptotic basis} if, in addition, $r_{A,\ul{b}}(n) = O_{\ul{b}}(1)$.

 For the unweighted form $\ul{b} = (1,\ldots,1)$, bounded asymptotic bases are the subject of the classical Erd\H{o}s--Tur\'an conjecture. Originally, this conjecture asserts that if $A\subseteq \N$ is an asymptotic basis for $x_1 + x_2$, then its representation function must be unbounded (see \cite{erdtur41}). More generally, one expects that no set $A\subseteq \N$ can be a bounded asymptotic basis for $x_1 + \cdots + x_h$. Despite considerable work, even the binary case remains open.

 The situation changes when unequal coefficients are allowed. For each integer $m\geq 2$, Moser \cite{mos62} constructed a set $A = A_m\subseteq\N$ such that $r_{A,(1,m)}(n)=1$ for every $n\in\N$. Nathanson \cite[Theorem 3]{nat07} extended this to the forms $x_1 + mx_2 + \cdots + m^{h-1}x_h$, showing that, for every $m,h\geq 2$, there exists $A = A_{m,h}\subseteq\N$ satisfying $r_{A,(1,m,\ldots,m^{h-1})}(n) = 1$ for every $n\in\N$.
 
 Thus the direct analogue of the Erd\H{o}s--Tur\'an conjecture is false for certain linear forms. However, the extent of this failure does not seem to have been systematically studied. In our previous work \cite[Conjecture 1.6]{taf25}, motivated by the threshold arising from a probabilistic model for prescribed representation functions, we conjectured that if $A\subseteq\N$ satisfies $r_{A,\ul{b}}(n)>0$ for every sufficiently large $n$, then
 \[ \limsup_{n\to\infty} \frac{r_{A,\ul{b}}(n)}{\log n} \geq 1. \]
 For $\ul{b} = (1,\ldots,1)$, this is a strong quantitative version of the Erd\H{o}s--Tur\'an conjecture. In the general weighted setting, however, Moser's and Nathanson's constructions already provide counterexamples with representation function identically equal to $1$. The present paper arose from examining how broadly this phenomenon persists.

\subsection{Bounded representation functions}
 Exact unique representation is a very rigid phenomenon. Kiss and S\'andor \cite{kissan25} recently characterized all strictly increasing coefficient vectors $\ul{b}$ and all proper subsets $A\subseteq\N$ for which $r_{A,\ul{b}}(n) = 1$ for every $n\in\N$. Their classification contains Moser's and Nathanson's constructions as basic examples. A related notion is that of a Sidon set for a linear form. Nathanson \cite{nat222} proved that an infinite $\ul{b}$-Sidon set exists precisely when the coefficients $b_1,\ldots,b_h$ have distinct subset sums. This concerns the injectivity of the linear form on $A^h$, but does not generally impose the basis property.

 A different line of research, originating in a question of S\'ark\"ozy and S\'os \cite{sarsoz97}, asks when the function $r_{A,\ul{b}}(n)$ can become constant for all sufficiently large $n$. Cilleruelo and Ru\'e \cite{cilrue09} settled the binary case, showing that apart from the Moser forms $(1,m)$, eventual constancy is impossible. Ru\'e and Spiegel \cite{ruespi20} proved that eventual constancy is impossible for broad families of coefficient vectors, including pairwise coprime coefficients $b_1,\ldots,b_h\geq 2$.

 In this paper, we study the intermediate problem of \emph{bounded}, rather than constant, representation functions. In the binary case, we show that every primitive $\ul{b} = (b_1,b_2) \neq (1,1)$ admits a bounded asymptotic basis. 

 \begin{thm}\label{MT1}
  Let $\ul{b} = (b_1,b_2) \neq (1,1)$ be a primitive vector of positive integers. Then there exists a set $A \subseteq\N$ such that
  \[ 1\leq r_{A,\ul{b}}(n)\leq (b_1 + 1)(b_2 + 1) \qquad (n\geq b_1b_2). \]
 \end{thm}

 Theorem \ref{MT1} shows that, among primitive $\ul{b} = (b_1,b_2) \in \Z_{\geq 1}$, the lower bound proposed in \cite[Conjecture 1.6]{taf25} can only possibly hold for the unweighted form.

\subsection{A general construction}
 The binary construction is a special case of a more general principle. Given finite subsets $S,T \subseteq \mathbb{Z}/H\mathbb{Z}$, we write $\mathbb{Z}/H\mathbb{Z} = S\oplus T$ if every residue class modulo $H$ can be written uniquely in the form $s+t$, with $s\in S$ and $t\in T$.

 \begin{thm}\label{MT2}
  Let $h\geq 2$, and let $\ul{b} = (b_1,\ldots,b_h)$ be a primitive vector of positive integers. Suppose that there exist a prime $p$, an integer $H\geq 1$, and a set $T \subseteq \mathbb{Z}/H\mathbb{Z}$ such that the residues $v_p(b_1),\ldots,v_p(b_h)\pmod H$ are pairwise distinct and
  \[ \mathbb{Z}/H\mathbb{Z} = \{v_p(b_1),\ldots,v_p(b_h)\} \oplus T. \]
  Then there exist a set $A\subseteq\N$, a constant $C_{\ul{b}}>0$, and an integer $n_0=n_0(\ul{b})$ such that
  \[ 1\leq r_{A,\ul{b}}(n)\leq C_{\ul{b}} \qquad (n\geq n_0). \]
 \end{thm}

 The direct sum condition in Theorem \ref{MT2} says that the valuations of the coefficients tile a finite cyclic group. A clean particular case is obtained when the valuations form an arithmetic progression.

 \begin{cor}\label{cor-val-ap}
  Let $h\geq 2$, and let $\ul{b} = (b_1,\ldots,b_h)$ be a primitive vector of positive integers. Suppose that, after permuting the coefficients, there exist a prime $p$ and $d\geq 1$ such that
  \[ v_p(b_i)\equiv(i-1)d\pmod{hd} \qquad (1\leq i\leq h). \]
  Then there exists $A\subseteq\N$ such that, for every sufficiently large $n$,
  \[ 1\leq r_{A,\ul{b}}(n) \ll_{\ul{b}} 1. \]
 \end{cor}

 In particular, Corollary \ref{cor-val-ap} applies to every primitive vector of the form
 \begin{equation}
  \ul{b}=(u_1, p^d u_2, p^{2d} u_3, \ldots, p^{(h-1)d}u_h),\qquad p\nmid u_1\cdots u_h. \label{ex-fam}
 \end{equation}
 Theorem \ref{MT2} is genuinely more general than this corollary. For example, for $\ul{b} = (1,p,p^4,p^5)$ the $p$-adic valuations are $\{0,1,4,5\}$, and $\mathbb{Z}/8\mathbb{Z} = \{0,1,4,5\}\oplus\{0,2\}$, so Theorem \ref{MT2} applies, whereas Corollary \ref{cor-val-ap} does not. The family \eqref{ex-fam} does not exhaust the corollary either: for instance, $\ul{b} = (1,p,p^5)$ is covered since $(v_p(b_1),v_p(b_2),v_p(b_3)) \equiv (0,1,2)\pmod 3$.
 
 Our results suggest that bounded representation is considerably more flexible than exact or eventually constant representation. It is natural to ask whether the hypothesis of Theorem \ref{MT2} can be generalized.

 \begin{prob}\label{gen-conj}
  Let $h\geq 2$. For which primitive vectors of positive integers $\ul{b} = (b_1,\ldots,b_h) \neq (1,\ldots,1)$ does there exist a bounded asymptotic basis for the linear form associated with $\ul{b}$?
 \end{prob}

 Simple cases not covered by our method include the forms associated with
 \[ \ul{b} = (1,1,2) \qquad\text{and}\qquad \ul{b} = (1,2,5). \]
 The first contains repeated coefficients, whereas the second has distinct subset sums.

 \begin{ntt}
  For $m\in\N$ and $A\subseteq\N$, we write $mA:=\{ma ~|~ a\in A\}$ for dilation. For a prime $p$ and $n\neq0$, we denote by $v_p(n)$ the largest $k\geq 0$ such that $p^k\mid n$.
 \end{ntt}

\section{Binary case}
 Let $p$ be a prime, let $e\geq 1$, and put $M:=p^e$. The basic set in our construction is
 \begin{equation}
  \mathcal{D} := \bigg\{\sum_{j\geq 0} \eps_j M^{2j} ~\bigg|~ \eps_j\in\{0,\ldots,M-1\},\ \eps_j = 0\text{ for all but finitely many } j\bigg\}. \label{setD}
 \end{equation}
 In base $p$, the digits of an element of $\mathcal{D}$ occupy the first $e$ positions in each block of length $2e$, while those of $M\mathcal{D}$ occupy the remaining $e$ positions. Thus every $k\in\N$ has a unique decomposition
 \begin{equation}
  k = z+Mw,\qquad z,w\in\mathcal{D}. \label{zMw}
 \end{equation}
 The same separation of digit positions gives the following injectivity property.

 \begin{lem}\label{digit-inj}
  If $\alpha,\beta\geq 1$ satisfy $v_p(\alpha)-v_p(\beta)\equiv e\pmod{2e}$, then the map
  \[ \mathcal{D}^2 \ni (z,w) \longmapsto \alpha z + \beta w\in\N \]
  is injective.
 \end{lem}
 \begin{proof}
  If $z,z'\in\mathcal{D}$ are distinct, then
  \[ v_p(z-z') \bmod{2e} \in \{0,\ldots,e-1\}. \]
  Indeed, if $\ell$ is the first position at which their base-$M^2$ digits differ, then $z-z' = M^{2\ell}(d+M^2q)$, where $q\in\mathbb{Z}$ and $0<|d|<M$. Hence $v_p(z-z') = 2e\ell + v_p(d)$, with $0\leq v_p(d)<e$.
  
  Now suppose that $\alpha z + \beta w = \alpha z'+\beta w'$. If $z\neq z'$, then also $w\neq w'$, and for some $0\leq i,j < e$,
  \begin{align*}
   v_p(\alpha(z-z'))-v_p(\beta(w'-w)) &= (v_p(\alpha)- v_p(\beta)) + v_p(z-z')-v_p(w'-w) \\
   &\equiv e+i-j \not\equiv 0 \pmod{2e}.
  \end{align*}
  But this contradicts $\alpha(z-z') = \beta(w'-w)$, so we must have $z=z'$ and $w=w'$.
 \end{proof}

 We now turn to the form $b_1x+b_2y$. Suppose first that the two variables were allowed to range over different sets. Together with \eqref{zMw}, the identity
 \begin{equation}
  b_1(b_2 z) + b_2(Mb_1 w) = b_1 b_2(z + Mw) \label{magic-id}
 \end{equation}
 shows that $b_2\mathcal{D}$ and $Mb_1\mathcal{D}$ naturally represent all multiples of $b_1b_2$. To cover the remaining residue classes, we add finitely many translates of these two sets; taking their union then gives a single set from which both variables may be chosen.

 \begin{proof}[Proof of Theorem \ref{MT1}]
  By symmetry, we may assume that $b_1>1$. Choose a prime $p\mid b_1$, and put $e := v_p(b_1)$ and $M := p^e$. Since $\gcd(b_1,b_2) = 1$, one has $p\nmid b_2$. Let $\mathcal{D}$ be the set from \eqref{setD}, and define
  \[ P_r := b_2\mathcal{D} + r\quad (0\leq r<b_2),\qquad Q_s := Mb_1\mathcal{D} + s\quad (0\leq s<b_1), \]
  and
  \begin{equation}
   A := \bigcup_{r=0}^{b_2-1} P_r\ \cup\ \bigcup_{s=0}^{b_1-1} Q_s. \label{setA}
  \end{equation}

  Put $B := b_1b_2$, and write $n = Bk+t$, where $k\geq 1$ and $0\leq t<B$. By the Chinese remainder theorem, there are unique $0\leq r< b_2$ and $0\leq s<b_1$ such that $b_1r + b_2s\equiv t\pmod B$, hence $b_1r + b_2s = t+qB$ for some $q\in\{0,1\}$. Writing $k-q = z+Mw$ with $z,w\in\mathcal{D}$, for
  \[ x := b_2z + r\in P_r,\qquad y := Mb_1 w + s\in Q_s, \]
  we have $b_1x + b_2y = B(z+Mw) + t+qB = n$. Thus $r_{A,(b_1,b_2)}(n)\geq 1$ for every $n\geq B$.

  For the upper bound, we distinguish the four possibilities for the sets containing $x, y \in A$. Once these sets and their indices are fixed, write $x,y$ in terms of elements $z,w\in\mathcal{D}$. The equation $b_1x+b_2y=n$ then becomes
  \[ n - b_1 r - b_2 s =
     \begin{cases}
      b_1 b_2 z + b_2^2 w,  &\text{if } x\in P_r,\ y\in P_s, \\
      b_1 b_2 z + Mb_1 b_2 w, &\text{if } x\in P_r,\ y\in Q_s, \\
      Mb_1^2 z + b_2^2 w, &\text{if } x\in Q_r,\ y\in P_s, \\
      Mb_1^2 z + Mb_1 b_2 w, &\text{if } x\in Q_r,\ y\in Q_s.                   
     \end{cases} \]
  In each ``$\alpha z + \beta w$'' on the right-hand side above, we have $v_p(\alpha) - v_p(\beta) \equiv e \pmod{2e}$, hence Lemma \ref{digit-inj} shows that, for fixed indices $r,s$, there is at most one choice of $(z,w)$. It remains to count the possible indices:
  \begin{itemize}
   \item If $x\in P_r$ and $y\in P_s$, reducing modulo $b_2$ gives $b_1r\equiv n\pmod{b_2}$, so $r$ is uniquely determined and $s$ has at most $b_2$ choices.\smallskip
   
   \item If $x\in P_r$ and $y\in Q_s$, reducing modulo $b_2$ determines $r$, while reducing modulo $b_1$ determines $s$, giving at most one choice.\smallskip
   
   \item If $x\in Q_r$ and $y\in P_s$, there are at most $b_1b_2$ choices altogether.\smallskip
   
   \item If $x\in Q_r$ and $y\in Q_s$, reducing modulo $b_1$ determines $s$, while $r$ has at most $b_1$ choices.
  \end{itemize}
  
  Therefore
  \[ r_{A,(b_1,b_2)}(n)\leq b_2+1+b_1b_2+b_1=(b_1+1)(b_2+1). \qedhere \]
 \end{proof}

\section{General construction}
 The construction below generalizes the decomposition used in the binary case. There, the two sets $\mathcal{D}$ and $M\mathcal{D}$ partition the base-$p$ positions into two classes, giving \eqref{zMw}. We now use a tiling of $\Z/H\Z$ to divide the digit positions among $h$ variables.

 Let $p$ be a prime, let $H\geq1$, and let $T \subseteq \mathbb{Z}/H\mathbb{Z}$. Define
 \[ E_T := \{j\in\N ~|~ j\bmod H\in T\} \]
 and
 \begin{equation}
  \mathcal{D}_T := \bigg\{\sum_{j\in E_T} \eps_j p^j ~\bigg|~ \eps_j \in \{0,\ldots,p-1\},\ \eps_j = 0\text{ for all but finitely many } j\bigg\}. \label{setDT}
 \end{equation}
 The digits of $\mathcal{D}_T$ occupy precisely the positions whose residues modulo $H$ belong to $T$. If the translates $a_i+T$ tile $\mathbb{Z}/H\mathbb{Z}$, the sets $p^{a_i}\mathcal{D}_T$ partition the base-$p$ positions in the same way that $\mathcal{D}$ and $M\mathcal{D}$ do in \eqref{zMw}.

 \begin{lem}\label{gen-digit-dec}
  Let $a_1,\ldots,a_h$ be pairwise distinct modulo $H$, and suppose that
  \[ \mathbb{Z}/H\mathbb{Z} = \{a_1,\ldots,a_h\}\oplus T. \]
  Put $K:=\max_{1\leq i\leq h} a_i$. Then every nonnegative multiple $m$ of $p^K$ can be written uniquely as
  \begin{equation}
   m = \sum_{i=1}^h p^{a_i} z_i,\qquad z_1,\ldots,z_h\in\mathcal{D}_T. \label{gen-z}
  \end{equation}
 \end{lem}
 \begin{proof}
  Write $m = \sum_{j\geq K} \eps_j p^j$ in base $p$. For each $j\geq K$, the direct sum hypothesis determines a unique index $i$ such that $j-a_i\in E_T$. We then place the digit $\eps_j$ in position $j-a_i$ of $z_i$. Thus the term $\eps_j p^j$ appears in $p^{a_i} z_i$, and summing over all $j$ gives \eqref{gen-z}. The decomposition is unique because the sets $a_i + E_T$ are pairwise disjoint.
 \end{proof}

 The same separation of digit positions gives the $h$-variable analogue of Lemma \ref{digit-inj}.

 \begin{lem}\label{gen-digit-inj}
  Under the hypotheses of Lemma \ref{gen-digit-dec}, let $\alpha_1,\ldots,\alpha_h \geq 1$ satisfy
  \[ v_p(\alpha_i)\equiv a_i\pmod H \qquad (1\leq i\leq h). \]
  Then the map
  \[ \mathcal{D}_T^h \ni (z_1,\ldots,z_h) \longmapsto \sum_{i=1}^h \alpha_i z_i \in \N \]
  is injective.
 \end{lem}
 \begin{proof}
  Suppose that $\sum_i \alpha_i z_i = \sum_i \alpha_i z_i'$. Whenever $z_i\neq z_i'$, the first base-$p$ position at which they differ belongs to $E_T$, and hence
  \[ v_p(\alpha_i (z_i-z_i'))\bmod H\in a_i+T. \]
  Since the sets $a_i+T$ are pairwise disjoint, the nonzero terms in $\sum_i \alpha_i (z_i-z_i') = 0$ have distinct valuations. This is impossible, since the least valuation in a vanishing sum must occur at least twice. Thus $z_i = z_i'$ for every $i$.
 \end{proof}

 We now apply this decomposition to the coefficients of the linear form.

 \begin{proof}[Proof of Theorem \ref{MT2}]
  Write
  \[ b_i = p^{a_i} u_i,\qquad a_i := v_p(b_i),\qquad p\nmid u_i \qquad (1\leq i\leq h). \]
  Let $L := \operatorname{lcm}(u_1,\ldots,u_h)$ and $d_i := L/u_i$. Since $p\nmid L$, we have $b_id_i = Lp^{a_i}$ for $1\leq i\leq h$. This gives the analogue of \eqref{magic-id}: if the variables were allowed to range over the different sets $d_i\mathcal{D}_T$, then we could write
  \begin{equation}
   \sum_{i=1}^h b_i(d_iz_i) = L\sum_{i=1}^h p^{a_i}z_i, \label{gen-magic}
  \end{equation}
  which, by Lemma \ref{gen-digit-dec}, represents all multiples of $Lp^{\max_i a_i}$. As in \eqref{setA}, we add finitely many translates and take their union.

  Put $K := \max_i a_i$ and $Q := Lp^K$. Since $\gcd(b_1,\ldots,b_h) = 1$, for each $0\leq r<Q$ choose integers $c_1(r),\ldots,c_h(r) \geq 0$ such that
  \[ \sum_{i=1}^h b_i c_i(r)\equiv r\pmod Q, \]
  and define
  \begin{equation}
   A := \bigcup_{r=0}^{Q-1} \bigcup_{i=1}^h\ (c_i(r) + d_i\mathcal{D}_T). \label{setAgen}
  \end{equation}
  Let $C(r) := \sum_i b_i c_i(r)$ and $n_0 := \max_r C(r)$. Given $n\geq n_0$, choose $r\equiv n\pmod Q$, so that $m := (n-C(r))/L$ is a nonnegative multiple of $p^K$. By Lemma \ref{gen-digit-dec}, we can write $m=\sum_i p^{a_i} z_i$ with $z_i\in\mathcal{D}_T$. Setting $x_i := c_i(r) + d_i z_i\in A$, \eqref{gen-magic} gives
  \[ \sum_{i=1}^h b_i x_i = C(r) + L\sum_{i=1}^h p^{a_i}z_i = C(r) + Lm = n, \]
  hence $r_{A,\ul{b}}(n)\geq 1$ for every $n\geq n_0$.

  For the upper bound, fix for each variable $x_k$ one of the $hQ$ sets appearing in \eqref{setAgen}, say $x_k = c_{j_k}(r_k) + d_{j_k}z_k$, where $z_k \in \mathcal{D}_T$. Then the equation $\sum_k b_k x_k = n$ becomes
  \[ \sum_{k=1}^h b_k d_{j_k} z_k = n - \sum_{k=1}^h b_k c_{j_k}(r_k). \]
  Once the sets containing the variables are fixed, the right-hand side is fixed. Moreover, since $p\nmid d_{j_k}$, we have $v_p(b_k d_{j_k}) = v_p(b_k) = a_k$. Lemma \ref{gen-digit-inj} therefore shows that there is at most one tuple $(z_1,\ldots,z_h) \in \mathcal{D}_T^h$ satisfying the equation. Since there are $hQ$ possible sets for each variable $x_k$, it follows that $r_{A,\ul{b}}(n) \leq (hQ)^h = (hLp^K)^h$.
 \end{proof}

 Corollary \ref{cor-val-ap} follows by taking $T := \{0,\ldots,d-1\} \subseteq \mathbb{Z}/hd\mathbb{Z}$.

\addtocontents{toc}{\protect\setcounter{tocdepth}{0}}
\section*{Acknowledgements}
 The author acknowledges the support of the São Paulo Research Foundation (FAPESP), Brazil, Process No.~2025/15961-3.
 
\addtocontents{toc}{\protect\setcounter{tocdepth}{1}}

\bibliographystyle{amsplain}
\bibliography{$HOME/Academie/Recherche/_latex/bibliotheca}%
\end{document}